\documentclass[a4paper,11pt]{article}

\usepackage[top=3.0cm, bottom=3.0cm, inner=3.0cm, outer=3.0cm,
includefoot]{geometry}

\usepackage{verbatim}
\usepackage{caption}
\usepackage{url}
\usepackage{amsmath}
\usepackage{geometry}
\usepackage{amssymb}
\usepackage{amsmath}
\usepackage{graphicx}
\usepackage{amsthm}
\usepackage{bbm}
\usepackage{float}
\usepackage{color,soul}
\usepackage{hyperref}
\usepackage[T1]{fontenc}
\usepackage[utf8]{inputenc}
\usepackage{authblk}
\usepackage{booktabs}
\usepackage{longtable}
\usepackage{enumerate}
\usepackage{tikz}
\usetikzlibrary{arrows}
\usetikzlibrary{decorations.markings}
\newcommand{\eChar}{\begin{enumerate}[(i)]}
\newcommand{\eCharR}{\begin{enumerate}[(a)]}
\newcommand{\eBr}{\begin{enumerate}[(1)]}

\title
{On Lichnerowicz sharp distance-regular graphs}

\author[1]{Kaizhe Chen}
\author[2]{Shiping Liu}
\author[3]{Heng Zhang}
\affil[1]{School of the Gifted Young, University of Science and Technology of China}
\affil[2,3]{School of Mathematical Sciences, University of Science and Technology of China} 
\affil[1]{ckz22000259@mail.ustc.edu.cn}
\affil[2]{spliu@ustc.edu.cn}
\affil[3]{hengz@mail.ustc.edu.cn}
\date{}

\theoremstyle{plain}
\newtheorem{lemma}{Lemma}[section]
\newtheorem{theorem}[lemma]{Theorem}

\newtheorem{corollary}[lemma]{Corollary}

\theoremstyle{definition}

\newtheorem{definition}[lemma]{Definition}
\newtheorem*{claim}{Claim}
\makeatletter
\renewenvironment{claim}[1][\unskip]{%
  \ifx#1\unskip
    \def\@thmopt{}%
  \else
    \def\@thmopt{#1 }%
  \fi
  \par\addvspace{\topsep}  % 添加垂直间距
  \noindent\textbf{Claim \@thmopt}\itshape
}{%
  \par\addvspace{\topsep}  % 结束时的间距
}
\makeatother

\newtheorem{remark}[lemma]{Remark}

\begin{document}

\maketitle

\begin{abstract}
The first non-zero Laplacian eigenvalue $\lambda_1$ of a finite graph is bounded below by its minimum Lin--Lu--Yau curvature $\kappa$. This is a discrete analogue of the classical Lichnerowicz Theorem. A graph with $\lambda_1=\kappa$ is called Lichnerowicz sharp.
In this note, we give a new proof of the classification of Lichnerowicz sharp distance-regular graphs, which was first obtained by Münch and strengthens the corresponding classification by Cushing, Kamtue, Koolen, Liu, M\"unch, and Peyerimhoff, which required an extra spectral condition.
As a key preparatory step, we provide a classification of all amply regular Terwilliger graphs with positive Lin--Lu--Yau curvature, a result that is interesting in its own right.
\end{abstract}

\section{Introduction}
Ricci curvature plays a central role in Riemannian geometry. Extending this idea to discrete settings has led to valuable perspectives and techniques in graph theory (see \cite{NR17} and the references therein). 
In this note, we focus on the Lin--Lu--Yau curvature for graphs \cite{LLY11}, which is built upon Ollivier's earlier work \cite{Ollivier09, Ollivier10}.
For two vertices $x$ and $y$, the Lin--Lu--Yau curvature $\kappa(x,y)$ compares the distance between $x$ and $y$ with the optimal transport distance between their respective neighborhoods (see Definition \ref{def:LLY}). 
This curvature notion has inspired numerous studies, see e.g. \cite{CLY, CKKLP20, DOJ19, M23, MS23, MW19, S22, Smith}.

Throughout this note, all graphs are assumed to be simple and finite.
The normalized Laplacian $L$ of a graph $G$ is defined as follows, 
\[L=I-D^{-\frac{1}{2}}AD^{-\frac{1}{2}},\] 
where $A$ stands for the adjacency matrix of $G$ and $D$ is the diagonal matrix of vertex degrees.
The following eigenvalue estimate in terms of Lin--Lu--Yau curvature is a discrete analogue of the well-known Lichnerowicz Theorem in Riemannian geometry.
\begin{theorem}[Discrete Lichnerowicz Theorem \cite{LLY11,Ollivier09}]\label{theorem:discreteBM} Let $G$ be a connected graph. Let $\lambda_1$ be the first nonzero eigenvalue of the normalized Laplacian $L$. 
    Then we have
    \begin{align}\label{Lsharp}
        \lambda_1\geq \min_{xy\in E}\kappa(x,y),
    \end{align}
   where $\kappa(x,y)$ is the Lin--Lu--Yau curvature of $xy$.
\end{theorem}

A graph is called {\it Lichnerowicz sharp} if equality is attained in the eigenvalue estimate of Theorem \ref{theorem:discreteBM}.
%(i.e., inequality \eqref{Lsharp} holds with equality).
A classical rigidity theorem of Obata \cite{Ob62} states that the only Lichnerowicz sharp Riemannian manifold is the round sphere.
A discrete counterpart of Obata's theorem on graphs has been established for Bakry–\'Emery curvature  \cite{LMP24}. However, the classification of Lichnerowicz sharp graphs for Lin--Lu--Yau curvature is still widely open.
The study of this problem for a highly symmetric class of graphs, called distance-regular graphs, was initiated in \cite{CKKLMP20}.

\begin{definition}[Distance-regular graph \cite{BCN89}]
  A connected graph $G$ with diameter $D$ is called \emph{distance-regular} if there are constants 
  $c_i, a_i, b_i$ such that for all $0\le i\le D$, 
  and all vertices $x$ and $y$ at distance $i$, among the neighbors of $y$, there are $c_i$ vertices at distance $i-1$ from $x$, $a_i$ vertices at distance $i$ from $x$, and $b_i$ vertices at distance $i+1$ from $x$.   
\end{definition}

It follows that a distance-regular graph $G$ is a regular graph with degree $d = b_0$, and that $c_i + a_i + b_i = d$ for all $i = 0, 1, \ldots, D$. 
The sequence $\{b_0, b_1, \ldots, b_{D-1}; c_1, c_2, \ldots, c_D\}$ is called the \emph{intersection array} of $G$.
In Figures \ref{fig:icosahedron} and \ref{fig:linepeter}, we present two examples of distance-regular graphs: the icosahedron, with intersection array $\{5,2,1; 1,2,5\}$, and the line graph of the Petersen graph, with intersection array $\{4,2,1; 1,1,4\}$.

\begin{figure}[H]
  \noindent\hspace*{-0.1cm}
  \begin{minipage}[t]{0.40\textwidth}
    \centering
    \includegraphics[width=0.55\linewidth]{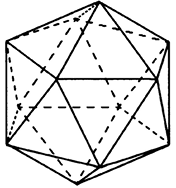}
    \caption{Icosahedron}
    \label{fig:icosahedron}
  \end{minipage}
  \hspace{0.05\textwidth}
  \begin{minipage}[t]{0.55\textwidth}
    \centering
    \includegraphics[width=0.46\linewidth]{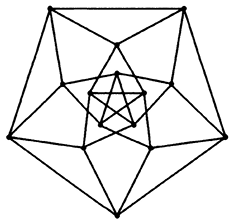}
    \caption{\mbox{Line graph of the Petersen graph}} 
    \label{fig:linepeter}
  \end{minipage}
\end{figure}

In \cite{CKKLMP20}, the authors classify the Lichnerowicz sharp distance-regular graphs under an additional assumption that $\theta_1=b_1-1$, where $\theta_1$ is the second largest adjacency eigenvalue.
In an arXiv preprint \cite[Theorems 3.6 and 4.2]{M}, M\"{u}nch shows that every Lichnerowicz sharp distance-regular graph is effective Bonnet--Myers sharp, and classifies all effective Bonnet--Myers sharp graphs.
In this note, we present a new proof of the classification of Lichnerowicz sharp distance-regular graphs.

\begin{theorem}\label{thm3}
   The Lichnerowicz sharp distance-regular graphs are precisely the following ones: the cocktail party graphs $CP(n)$ for $n\ge 2$, the Hamming graphs $H(d,n)$ for $d\ge 1$ and $n\ge 2$, the Johnson graphs $J(n,k)$ for $1\le k\le n-1$, the demi-cubes $Q^n_{(2)}$ for $n\ge 2$, the Schl\"afli graph, and the Gosset graph. 
\end{theorem}

We point out that our proof of Theorem \ref{thm3} is completely different from that of M\"{u}nch, and highlights a further connection between distance-regular graphs and effective Bonnet--Myers sharp graphs.
We also remark that the Cartesian products of Lichnerowicz sharp distance-regular graphs coincide with
the complete connected hypermetric spaces classified by Terwilliger and Deza \cite{T}.
A key step of our proof of Theorem \ref{thm3} is the classification of all amply regular Terwilliger graphs with positive Lin--Lu--Yau curvature (Theorem \ref{ARTG}).

\begin{definition}[Amply regular graph \cite{BCN89}]
    A non-complete connected $d$-regular graph on $n$ vertices is called {\it amply regular} with parameters $(n,d,\alpha,\beta)$ if the following two conditions hold:
    \begin{itemize}
        \item[(i)] any two adjacent vertices have exactly $\alpha$ common neighbors;
        \item[(ii)] any two vertices at distance $2$ have exactly $\beta$ common neighbors.
    \end{itemize}
\end{definition}

Compared to the definition of amply regular graphs given in \cite{BCN89}, we additionally assume that the graph is non-complete and connected for the convenience of discussion.
%For convenience, all amply regular graphs are assumed to be connected and non-complete throughout this note.
For such graphs, the Lin--Lu--Yau curvature has proven to be a powerful tool for bounding their diameters and eigenvalues \cite{CHLZ24,CL,HLX24}.

We note that any non-complete distance-regular graph is amply regular. In fact, a distance-regular graph on $n$ vertices with intersection array
$\{b_0, b_1, \ldots, b_{D-1}; c_1, c_2, \ldots, c_D\}$ is an amply regular graph with parameters $(n,d,\alpha, \beta)=(n,b_0,b_0-b_1-1,c_2)$.
Consequently, when discussing a non-complete distance-regular graph, we often identify $d=b_0$, $\alpha =b_0-b_1-1$, and $\beta= c_2$.

A non-complete graph $G$ is called a {\it Terwilliger} graph if, for any two vertices $x$ and $y$ at distance 2, the subgraph of $G$ induced by the common neighbors of $x$ and $y$ is a clique of size $\beta$ (for some fixed $\beta \geq 1$).
Note that an amply regular graph is Terwilliger if and only if it contains no induced quadrangles.

\begin{theorem}\label{ARTG}
    Let $G$ be an amply regular Terwilliger graph with positive Lin--Lu--Yau curvature. Then, $G$ 
    is isomorphic to one of the following graphs: the pentagon, the icosahedron, the line graph of the Petersen graph, the line graph of the Hoffman–Singleton graph, or the line graph of a strongly regular graph with parameters $(3250,57,0,1)$.
\end{theorem}

Classifying graphs with positive Lin–Lu–Yau curvature under various natural constraints is a fruitful research direction. Examples include classifications for graphs that are $C_4$-free \cite{LY,LY24}, outerplanar \cite{BOSWY24,LLW}, and Halin graphs \cite{CLLY}. %Perhaps one of the most interesting open problems in this direction is the classification of planar graphs with positive Lin–Lu–Yau curvature (and minimum degree at least 3).
Amply regular graphs with parameters $(n,d,\alpha,\beta)$ were proved to have positive Lin--Lu--Yau curvature whenever $\alpha=\beta >1$ \cite{LL21} or $1\ne \beta> \alpha$ \cite{HLX24}. 
A direct consequence of Theorem \ref{ARTG} is the classification of all amply regular graphs with positive Lin--Lu--Yau curvature and $\beta =1$ (Corollary \ref{mu1}).

The rest of the paper is organized as follows. In Section \ref{section2}, we recall the necessary background on the Lin--Lu--Yau curvature, amply regular graphs, strongly regular graphs, distance-regular graphs, and Terwilliger graphs. In Section \ref{section3} and Section \ref{section4}, we prove Theorem \ref{ARTG} and Theorem \ref{thm3}, respectively.

\section{Preliminaries}\label{section2}
\subsection{Ollivier Ricci curvature and Lin--Lu--Yau curvature}\label{sec1}
%The Wasserstein distance $W_1$ is a fundamental concept from optimal transport theory. 
Intuitively, the Wasserstein distance quantifies the minimal total cost of transporting one probability measure to another, where the cost of transporting a unit of mass between two vertices is given by their combinatorial distance. 
%This makes $W_1$ the natural metric for comparing the neighborhood measures in the definition of Lin--Lu--Yau curvature.

\begin{definition}[Wasserstein distance]
   
Let $G=(V,E)$ be a locally finite graph, $\mu_1$ and $\mu_2$ be two probability measures on $G$. The Wasserstein distance $W_1(\mu_1, \mu_2)$ between $\mu_1$ and $\mu_2$ is defined as
\begin{align}\label{W}
    W_1(\mu_1,\mu_2)=\inf_{\pi}\sum_{y\in V}\sum_{x\in V}d(x,y)\pi(x,y),
\end{align}
where $d(x,y)$ denotes the combinatorial distance between $x$ and $y$ in $G$, and the infimum is taken over all maps $\pi: V\times V\to [0,1]$ satisfying
$$\mu_1(x)=\sum\limits_{y\in V}\pi(x,y) \text{ for any $x\in V$ and }\mu_2(y)=\sum\limits_{x\in V}\pi(x,y) \text{ for any $y\in V$}.$$
Such a map $\pi$ is called a {\it transport plan}.
A transport plan $\pi$ that attains the infimum in \eqref{W} is called an {\it optimal} transport plan.
We call a transport plan {\it simple} if, for each vertex $x\in V$, we have
$$\pi(x,x)=\min\{ \mu_1(x), \mu_2(x)\}.$$
\end{definition}
We consider the following particular probability measure around a vertex $x\in V$:
     \[\mu_x^p(y)=\left\{
                    \begin{array}{ll}
                      p, & \hbox{if $y=x$;} \\
                      \frac{1-p}{\mathrm{deg}(x)}, & \hbox{if $yx\in E$;} \\
                      0, & \hbox{otherwise,}
                    \end{array}
                  \right.
     \]
     where the vertex degree $\mathrm{deg}(x)$ is the number of edges incident with $x$.

\begin{definition}[$p$-Ollivier curvature \cite{Ollivier09} and Lin--Lu--Yau curvature \cite{LLY11}]\label{def:LLY} Let $G=(V,E)$ be a locally finite graph. For any vertices $x,y\in V$, the $p$-Ollivier curvature $\kappa_p(x,y)$, $p\in [0,1]$, is defined as
       \[\kappa_p(x,y)=1-\frac{W_1(\mu_x^p,\mu_y^p)}{d(x,y)}.\]
The Lin--Lu--Yau curvature $\kappa(x,y)$ is defined as
       \[\kappa(x,y)=\lim_{p\to 1}\frac{\kappa_p(x,y)}{1-p}.\]
       \end{definition}
Notice that $\kappa_1(x,y)$ is always $0$. Hence, the Lin--Lu--Yau curvature $\kappa(x,y)$ equals  the negative of the left derivative of the function $p\mapsto \kappa_p(x,y)$ at $p=1$.

Our later proof is closely related to the following upper bound on the Lin--Lu--Yau curvature of a regular graph.
\begin{lemma}[{\cite[Proposition 2.7]{CKKLMP20}}]
\label{up}
    Let $G=(V,E)$ be a $d$-regular graph. For any edge $xy\in E$, we have
\begin{align}
    \kappa(x,y)\leq \frac{2+|N_{xy}|}{d},
\end{align}
where $N_{xy}:=\{z\in V\ |\ xz\in E, yz\in E\}$.
\end{lemma}
   
The following lemma is quite useful in calculating the Wasserstein distance and the Lin--Lu--Yau curvature.

\begin{lemma}[{\cite[Lemma 4.1]{CLY}\label{simple}}]
    Let $G$ be a connected graph. Let $x$ and $y$ be two adjacent vertices in $G$ with $d_x\ge d_y$. For any $p\in \left[ \frac{1}{1+d_y},1 \right]$, there is a simple optimal transport plan from $\mu_x^p$ to $\mu_y^p$ such that $$\pi(x,y)=p-\frac{1-p}{d_y}.$$ 
\end{lemma}

\subsection{Amply regular graphs and strongly regular graphs}
\iffalse
\begin{definition}[Amply regular graph \cite{BCN89}] Let $G$ be a $d$-regular graph with $n$ vertices which is neither complete nor empty. Then  $G$ is called an amply regular graph with parameters $(n,d,\alpha,\beta)$ if the
following two conditions hold:
        \begin{itemize}
           \item [(i)] Any two adjacent vertices have exactly $\alpha$ common neighbors;
           \item [(ii)] Any two vertices at distance $2$ have exactly $\beta$ common neighbors.
        \end{itemize}     
\end{definition}
\fi

The following result gives a basic inequality for the parameters of an amply regular graph.

\begin{theorem}[{\cite[Theorem 1.2.3]{BCN89}}]\label{lem:d-lowbdd}
    Let $G$ be an amply regular graph with parameters $(n, d, \alpha, \beta)$.
\begin{enumerate}
    \item[(i)] We have
    \begin{equation}\label{eq:7}
        d \geq 2\alpha + 3 - \beta.
    \end{equation}
    %and, with $b_2 = \min\{b_2(\alpha, \beta) \mid d(\alpha, \beta) = 2\}$, every $2$-claw is in at most $\max\{0, k - 2\lambda - 3 + \mu - b_2\}$ quadrangles.
    
    \item[(ii)] Equality holds in \eqref{eq:7} if and only if $G$ is the icosahedron or the line graph of a regular graph of girth at least $5$. In particular, if $G$ has diameter 2, then $G$ is the pentagon.
\end{enumerate}
\end{theorem}

An amply regular graph with diameter 2 is called a {\it strongly regular} graph.
The eigenvalues of a strongly regular graph can be determined by its parameters.

\begin{theorem}[{\cite[Chapter 10]{GR01}}]\label{lam}
    Let $G$ be a strongly regular graph with parameters $(n,d,\alpha,\allowbreak \beta)$. Then 
    $G$ has exactly three distinct Laplacian eigenvalues:
          \begin{align*}
            &\lambda_0=0,\\
            &\lambda_1=1-\frac{(\alpha-\beta)+\sqrt{(\alpha-\beta)^2+4(d-\beta)}}{2d},\\
            &\lambda_{n-1}=1-\frac{(\alpha-\beta)-\sqrt{(\alpha-\beta)^2+4(d-\beta)}}{2d}.
          \end{align*}
%The corresponding multiplicities are 
%\begin{align*}
  %&m_1=1,\\
  %&m_2=\frac{1}{2}\biggl((n-1)-\frac{2d+(n-1)(\alpha-\beta)}{\sqrt{(\alpha-\beta)^2+4(d-\beta)}}\biggr),\\
  %&m_3=\frac{1}{2}\biggl((n-1)+\frac{2d+(n-1)(\alpha-\beta)}{\sqrt{(\alpha-\beta)^2+4(d-\beta)}}\biggr).
%\end{align*}
\end{theorem}

\begin{remark}
    Let $G$ be a connected $d$-regular graph on $n$ vertices. Then the adjacency eigenvalues $d=\theta_0> \theta_1\ge \theta_2\ge \cdots \ge \theta_{n-1}$ and the (normalized) Laplacian eigenvalues $0=\lambda_0< \lambda_1\le \lambda_2\le \cdots \le \lambda_{n-1}$ of $G$ satisfy $\theta_i= d-d\lambda_i$ for $0\le i\le n-1$.
\end{remark}

The parameters of a strongly regular graph are related by the following identity.

\begin{theorem}[{\cite[Chapter 10]{GR01}}]\label{lem:param_relation}
    Let $G$ be a strongly regular graph with parameters $(n, d, \alpha,\allowbreak \beta)$. Then its parameters satisfy the following relation:
    \[
    n - d - 1 = \frac{d(d - \alpha - 1)}{\beta}.
    \]
\end{theorem}

We conclude this subsection by presenting two well-known classification results on strongly regular graphs, where the first one is due to Seidel. 

\begin{theorem}[{\cite[Theorem 1.1.1]{BV22}}]\label{thm:1.1.1}
  A strongly regular graph with the least adjacency eigenvalue $-2$ is one of the examples in (i)–(viii):
\begin{enumerate}
\renewcommand{\labelenumi}{\textup{(\roman{enumi})}}

\item the cocktail party graphs $CP(n)$, with parameters $(2n,2n-2,2n-4,2n-2)$ for $n \geq 2$;

\item the Hamming graphs $H(2,n)$, with parameters $(n^{2},2(n-1),n-2,2)$ for $n \geq 2$;

\item the triangular graphs $T(n)$, with parameters $\bigl(\binom{n}{2},2(n-2),
n-2,4\bigr)$ for $n \geq 4$;

\item the Shrikhande graph, with parameters $(16,6,2,2)$;

\item the three Chang graphs, with parameters $(28,12,6,4)$;

\item the Petersen graph, with parameters $(10,3,0,1)$;

\item the Clebsch graph, with parameters $(16,10,6,6)$;

\item the Schläfli graph, with parameters $(27,16,10,8)$.

\end{enumerate}
\end{theorem}

\begin{theorem}[{\cite{HS}}] \label{thm:HS_Moore_srg}
Let $H$ be a strongly regular graph with parameters $(n,d,0,1)$.
Then $d\in\{2,3,7,57\}$ and $n=d^{2}+1$.
Moreover,
\begin{itemize}
\item if $d=2$, then $H$ is the pentagon;
\item if $d=3$, then $H$ is the Petersen graph;
\item if $d=7$, then $H$ is the Hoffman--Singleton graph.
\end{itemize}
\end{theorem}

In particular, we mention that the existence of a strongly regular graph with parameters $(3250,57,0,\allowbreak 1)$ is a long-standing and well-known open problem in algebraic graph theory.

\subsection{Distance-regular graphs}
In this subsection, we recall several known results on distance-regular graphs that will be used in our proof.

\begin{theorem}[{\cite[Theorem 4.4.3 (ii)]{BCN89}}]\label{thm:4.4.3}
 Let $G$ be a distance-regular graph of diameter $D\ge 3$ with the second largest eigenvalue $\theta_1$, and put
$b^+ = \frac{b_1}{\theta_1 + 1}$.
If $b^+ < 1$, then either $\beta = 1$ or $G$ is the icosahedron.
%Then $b^+ > 0$, $b^- < -1$, and we have:
    %\item[(i)] Each neighbourhood $\Gamma(\gamma)$ is a graph with smallest eigenvalue $\geq -1 - b^+$ and second largest eigenvalue $\leq -1 - b^-$ (here the second largest eigenvalue is taken to be the valency $\lambda$ in case $\Gamma(\gamma)$ is disconnected).
    %\item[(ii)] If $b^+ < 1$ then either $\mu = 1$ or $\Gamma$ is the icosahedron. 
    %\item[(iii)] If $b^- > -2$ then either $\lambda = 0$ or $\Gamma$ is the icosahedron.
\end{theorem}   

\begin{remark}
    The condition $b^+ < 1$ in Theorem \ref{thm:4.4.3} is equivalent to $\lambda_1< (\alpha +2)/d$.
\end{remark}

We conclude this subsection by presenting two classical classification results on distance-regular graphs, the second of which is due to Hall.

\begin{theorem}[{\cite[Theorem 4.4.11]{BCN89}}]\label{thm:4.4.11}
Let $G$ be a distance-regular graph with second largest eigenvalue $\theta_1 = b_1 - 1$. Then at least one of the following holds\textup{:}
\begin{enumerate}
    \item[(i)] $G$ is a strongly regular graph with least eigenvalue $-2$;
    
    \item[(ii)] $\beta = 1$;
    
    \item[(iii)] $\beta = 2$, and $G$ is a \textit{Hamming graph}, a \textit{Doob graph}, the Doro graph, or the Conway--Smith graph;
    
    \item[(iv)] $\beta = 4$, and $G$ is a \textit{Johnson graph};
    
    \item[(v)] $\beta = 6$, and $G$ is a \textit{demi-cube};
    
    \item[(vi)] $\beta = 10$, and $G$ is the \textit{Gosset graph}.
\end{enumerate}
\end{theorem}

We remark that the Doro graph and the Conway--Smith graph are both locally Petersen graphs. Indeed, there are only three locally Petersen graphs \cite{H}, namely the Doro graph,  the Conway--Smith graph, and the complement of the triangular graph $T(7)$. The Doro graph is also known as the Hall graph. In this paper, we follow the terminology of \cite{BCN89} and refer to it as the Doro graph; see \cite[Section 12.2]{BCN89} for further details.

%\begin{theorem}[{\cite[Theorem 1.16.5]{BCN89}}]\label{thm:1.16.5} 
%There are up to isomorphism precisely three connected locally Petersen graphs, namely:
%\begin{enumerate}
    %\item[(i)] The unique distance-regular graph with $21$ vertices and intersection array $\{10,6; 1,6\}$: the complement of the triangular graph $T(7)$.
    
    %\item[(ii)] The unique distance-regular graph with $63$ vertices and intersection array $\{\allowbreak 10,\allowbreak 6,\allowbreak 4,\allowbreak 1;\allowbreak 1,\allowbreak 2,\allowbreak 6,\allowbreak 10\}$: the Conway--Smith graph.
    
    %\item[(iii)] The unique distance-regular graph with $65$ vertices and intersection array $\{\allowbreak 10,\allowbreak 6,\allowbreak 4; \allowbreak 1,\allowbreak 2,\allowbreak 5\}$: the Doro graph.
%\end{enumerate}
%\end{theorem}

\subsection{Terwilliger graphs}
We first present the following observation.

\begin{lemma}\label{beta=1}
    An amply regular graph $G$ with parameters $(n,d,\alpha,1)$ is a Terwilliger graph.
\end{lemma}
\begin{proof}
    Let $u$ and $v$ be two vertices at distance 2 in $G$.
    Since $G$ has parameters $(n,d,\alpha,1)$,
    $u$ and $v$ have exactly one common neighbor.
    That is, the subgraph of $G$ induced by the common neighbor of $u$ and $v$ is a single vertex. So, $G$ is a Terwilliger graph.
\end{proof}

We introduce the following concepts, see \cite{BCN89}, for instance. For a vertex $x$ of a graph $G=(V,E)$, let $G(x)$ be the subgraph of $G$ induced by the neighbors of $x$.
We write $x^{\perp}$ for the set of vertices consisting of $x$ and its neighbors. 
Let $x\equiv y$ if $x^{\perp}=y^{\perp}$, then $\equiv$ is an equivalence relation, and we shall write $\overline{x}$ for the equivalence class of the vertex $x$. Consider an edge $xy\in E$. Then, for any $x'\in \overline{x}$ and $y'\in \overline{y}$, we have $x'y'\in E$. Therefore, the following graph is well-defined: We write $\overline{G}$ for the quotient $G/\equiv$, which has vertices  $\overline{x}$ for $x\in V$, and $\overline{x}$ is adjacent to $\overline{y}$ whenever $xy\in E$ and $\overline{x}\neq\overline{y}$. $\overline{G}$ is called the \emph{reduced graph} of $G$, and $G$ is called \emph{reduced} when all equivalence classes have size one. 

Terwilliger graphs possess many striking properties. Here, we present two of them that are crucial for our later discussion. 

\begin{lemma}[{\cite[Corollary 1.16.6 (i)]{BCN89}}]\label{cor:1.16.6}
Let $G$ be a strongly regular Terwilliger graph with parameters $(n,d,\alpha,\beta)$.
If $n < 50\beta$, then $G$ is the pentagon or the Petersen graph.
%\begin{enumerate}
%\item[\textup{(i)}] If $d = 2$ and $v < 50\mu$, then $G$ is the pentagon or the Petersen graph.
%\item[\textup{(ii)}] If $k < 50(\mu-1)$, then $G$ is the icosahedron \textup{(}$v = 12$, $d = 3$\textup{)} or a locally Petersen graph \textup{(}and then $v = 63$, $d = 4$ or $v = 65$, $d = 3$\textup{)}.
%\end{enumerate}    
\end{lemma}

\begin{theorem}[{\cite[Theorem 1.16.3]{BCN89}}]\label{thm:1.16.3}
Let $G$ be an amply regular Terwilliger graph with parameters $(n,d,\alpha,\beta)$, $\beta > 1$. Then, for any $\gamma \in G$, the reduced graph $\overline{G(\gamma)}$ of $G(\gamma)$ is a strongly regular Terwilliger graph with parameters
\begin{align*}
\bar{n} &= \frac{d}{s}, \quad \bar{d} = \frac{\alpha - s + 1}{s}, \quad \bar{\beta} = \frac{\beta - 1}{s}, \\
\bar{\alpha} &= \frac{(\alpha - s + 1)(\alpha - 2s + 1) - (\beta - 1)(d - \alpha - 1)}{s(\alpha - s + 1)},
\end{align*}
where $s$ is the size of the equivalence classes of $G(\gamma)$. Moreover, 
\begin{gather*}
\beta = s + 1 \text{ or } \beta \geq s^{2} + s + 1, \\
s \mid \gcd(d, \alpha + 1, \beta - 1), \quad \alpha - s + 1 \mid (\beta - 1)(d - s), \quad s \leq \bar{\alpha} + 1.
\end{gather*}
\end{theorem}

\begin{remark}
    Let $G$ be a connected regular Terwilliger graph of diameter 2. By \cite[Proposition 1.16.2]{BCN89}, all equivalence classes $\bar{x}$ ($x\in G$) have the same cardinality. Thus, the parameter $s$ in Theorem \ref{thm:1.16.3} is well-defined.
\end{remark}

\section{Amply regular Terwilliger graphs with positive curvature}\label{section3}
In this section, we prove Theorem \ref{ARTG}. We first show three preparatory lemmas. Let $G=(V,E)$ be a graph and $xy\in E$ be an edge $G$. 
Recall from Section 2 that $N_{xy}$ denotes the set of common neighbors of $x$ and $y$.
Let
$$A^{(xy)}_{x}:=\{ w\in V\ |\ w\ne y,\ wx\in E,\ wy\notin E \},$$
    and
$$A^{(xy)}_{y}:=\{ w\in V\ |\ w\ne x,\ wy\in E,\ wx\notin E\}.$$
For a fixed edge $xy$, we simply write $A_x$ for $A^{(xy)}_{x}$ and write $A_y$ for $A^{(xy)}_{y}$ when there is no risk of confusion.

\begin{lemma}\label{ARTG_up}
    Let $G=(V,E)$ be an amply regular Terwilliger graph with parameters $(n, d, \alpha, \beta)$.
    For any edge $xy\in E$, the Lin--Lu--Yau curvature of $xy$ satisfies
    $$\kappa(x,y) \le \frac{2\alpha +3-d}{d}.$$
\end{lemma}
\begin{proof}
For any edge $xy\in E$, we have  $|A_x|=|A_y|=d-\alpha -1$.
    For any $p\in \left[ \frac{1}{1+d},1 \right]$,
    by Lemma \ref{simple}, there is a simple optimal transport plan $\pi$ from $\mu_x^p$ to $\mu_y^p$ such that $$\pi(x,y)=p-\frac{1-p}{d}.$$ 
    For any two vertices $w_1\in A_x$ and $w_2\in A_y$, if $w_1$ is adjacent to $w_2$, then $xyw_2w_1x$ forms an induced quadrangle, which is contradictory to the assumption that $G$ is a Terwilliger graph. Thus, for any two vertices $w_1\in A_x$ and $w_2\in A_y$, we have $d(w_1,w_2)\ge 2$.
    By the choice of $\pi$, we observe that, for any two vertices $w_1$ and $w_2$ in $G$, if $\pi(w_1,w_2)>0$, then either ``$w_1=x$ and $w_2=y$'' or ``$w_1\in A_x$ and $w_2\in A_y$''.
    We calculate that
    \begin{align*}
        W_1(\mu_x^p,\mu_y^p)&=\sum_{w_1\in V}\sum_{w_2\in V}d(w_1,w_2)\pi(w_1,w_2)\\
        &=p-\frac{1-p}{d}+\sum_{w_1\in A_x}\sum_{w_2\in A_y} d(w_1,w_2)\pi(w_1,w_2)\\
        &\ge p-\frac{1-p}{d}+2\sum_{w_1\in A_x}\sum_{w_2\in A_y}\pi(w_1,w_2)\\
        &= p-\frac{1-p}{d}+ \frac{2(d-\alpha-1)(1-p)}{d}.
    \end{align*}
    By the definition of the Lin--Lu--Yau curvature, we derive
    \begin{align*}
        \kappa(x,y)=\lim_{p\to 1}\frac{1-W_1(\mu_x^p,\mu_y^p)}{1-p} \le \frac{2\alpha +3-d}{d}.
    \end{align*}
    This completes the proof.
\end{proof}

Let $H$ be a $d$-regular graph on $n$ vertices of girth at least $5$. It is direct to check that the line graph $G$ of $H$ is amply regular with parameters $(\frac{nd}{2},2d-2,d-2,1)$. Hence, $G$ is an amply regular Terwilliger graph by Lemma \ref{beta=1}. Next, we show that for such a graph $G$, the curvature upper bound in Lemma \ref{ARTG_up} can be improved for at least one edge if the diameter of $H$ is at least $3$.

\begin{lemma}\label{line}
    Let $H$ be a regular graph of girth at least $5$ with diameter at least 3, and $G$ be its line graph. Then there exists an edge of $G$ with non-positive Lin--Lu--Yau curvature.
\end{lemma}
\begin{proof}
    Let $u_1$ and $u_4$ be two vertices at distance 3 in $H$, and let $u_1u_2u_3u_4$ be a path of length 3 connecting $u_1$ and $u_4$.
    Let $x,y,z$ be the vertices in $G$ corresponding to the edges $u_1u_2,u_2u_3,u_3u_4$ in $H$, respectively.
    \iffalse
    Let $V(G)$ the vertex set of $G$, and let $E(G)$ the edge set of $G$.
    Set 
    $$N_{xy}:=\{ w\in V(G)\ |\ wx\in E(G),\ wy\in E(G) \},$$ 
    $$A_{x}:=\{ w\in V(G)\ |\ w\ne y,\ wx\in E(G),\ wy\notin E(G) \},$$
    and
    $$A_{y}:=\{ w\in V(G)\ |\ w\ne x,\ wy\in E(G),\ wx\notin E(G) \}.$$
    \fi
    \leavevmode
    \iffalse
    \begin{claim}[1]
        For any two vertices $w_1\in A_x$ and $w_2\in A_y$, we have $d(w_1,w_2)\ge 2$.
    \end{claim}
    \begin{proof}
        Let $u_1u_5$ and $u_3u_6$ be the two edges in $H$ corresponding to $w_1$ and $w_2$, respectively.
        Since $A_x$ and $A_y$ are disjoint, we have $w_1\ne w_2$.
        If $d(w_1,w_2)=1$, then $u_5= u_6$.
        Note that $u_1u_2u_3u_5u_1$ is a cycle of length 4, contradicting the assumption that $H$ has girth at least 5.
        Therefore, we have $d(w_1,w_2)\ge 2$.
    \end{proof}
\fi

        We claim that, for any vertex $w\in A_x$, we have $d(w,z)\ge 3$.

        Let $u_1u_5$ be the edge in $H$ corresponding to the vertex $w$ in $G$.
        Let $E(H)$ be the edge set of $H$. Since $d(u_1,u_4)=3$, we have $u_1u_3\notin E(H)$ and $u_1u_4\notin E(H)$, and hence $u_5\notin \{u_3,u_4 \}$.
        Thus, $d(w,z)\ge 2$.

        Since $d(u_1,u_4)=3$, we have $u_5u_4\notin E(H)$.
        If $d(w,z)=2$, then $u_5u_3\in E(H)$.
        Note that $u_1u_2u_3u_5u_1$ is a cycle of length 4, contradicting the assumption that $H$ has girth at least 5.
        Thus, we have $d(w,z)\ge 3$, as claimed.

    Suppose that $H$ is $d$-regular. Then $G$ is $(2d-2)$-regular.
    For any $p\in \left[ \frac{1}{2d-1},1 \right]$,
    by Lemma \ref{simple}, there is a simple optimal transport plan $\pi$ from $\mu_x^p$ to $\mu_y^p$ such that $$\pi(x,y)=p-\frac{1-p}{2d-2}.$$ 
    Notice that $d(w_1,w_2)\geq 2$ for any $w_1\in A_x$ and $w_2\in A_y$ since $G$ is Terwilliger. 
    %By the choice of $\pi$, we observe that, for any two vertices $w_1$ and $w_2$ in $G$, if $\pi(w_1,w_2)>0$, then either ``$w_1=x$ and $w_2=y$'' or ``$w_1\in A_x$ and $w_2\in A_y$''.
    Then, it follows by the above Claim that
    \begin{align*}
        W_1(\mu_x^p,\mu_y^p)&=\sum_{w_1\in V}\sum_{w_2\in V}d(w_1,w_2)\pi(w_1,w_2)\\
        &=p-\frac{1-p}{2d-2}+\sum_{w_1\in A_x}d(w_1,z)\pi(w_1,z)+\sum_{w_1\in A_x}\sum_{w_2\in A_y\backslash \{z\}}d(w_1,w_2)\pi(w_1,w_2)\\
        &\ge p-\frac{1-p}{2d-2}+3\sum_{w_1\in A_x}\pi(w_1,z)+2\sum_{w_1\in A_x}\sum_{w_2\in A_y\backslash \{z\}}\pi(w_1,w_2)\\
        &= p-\frac{1-p}{2d-2}+\frac{3(1-p)}{2d-2} + \frac{2(|A_y|-1)(1-p)}{2d-2}.
    \end{align*}
    Since $|N_{xy}|=d-2$, we have $|A_y|=2d-2-|N_{xy}|-1=d-1$.
    Thus,
    \begin{align*}
        W_1(\mu_x^p,\mu_y^p)
        \ge \left(p-\frac{1-p}{2d-2}\right)+\frac{3(1-p)}{2d-2} + \frac{2(d-2)(1-p)}{2d-2}=1.
    \end{align*}
    It follows that
    \begin{align*}
        \kappa(x,y)=\lim_{p\to 1}\frac{1-W_1(\mu_x^p,\mu_y^p)}{1-p} \le 0.
    \end{align*}
    This completes the proof.
\end{proof}

Let $H$ be a $d$-regular graph on $n$ vertices with girth at least $5$. If the diameter of $H$ equals $2$, then $H$ is a strongly regular graph with parameters $(n,d,0,1)$. Next, we show for such a graph $H$, the curvature upper bound of its line graph $G$ in Lemma \ref{ARTG_up} is achieved.

\begin{lemma}\label{3.3}
    Let $G=(V,E)$ be the line graph of a strongly regular graph $H$ with parameters $(n,d,0,1)$. Then for any edge $xy$ in $G$, we have
    $$\kappa(x,y)=\frac{1}{2d-2}.$$
\end{lemma}
\begin{proof}
    Note that $H$ has girth at least 5.
    It is straightforward to check that $G$ is an amply regular graph with parameters $(nd/2,2d-2,d-2,1)$.
    By Lemma \ref{beta=1}, $G$ is a Terwilliger graph.
    It follows by Lemma \ref{ARTG_up} that $\kappa(x,y) \le 1/(2d-2)$.
    Now, it suffices to show that $\kappa(x,y) \ge 1/(2d-2)$.

    For any edge $xy$ of $G$, let $u_1u_2$ and $u_1u_3$ be the edges in $H$ corresponding to $x$ and $y$, respectively.
    Let $V(H)$ be the vertex set of $H$, and let $E(H)$ be the edge set of $H$.
    Set $$A_2:=\{ v\in V(H)\ |\ v\ne u_1,\ u_2 v\in E(H) \}$$
    and
    $$A_3:=\{ v\in V(H)\ |\ v\ne u_1,\ u_3 v\in E(H) \}.$$
    Observe that $A_2\cap A_3=\emptyset$. Indeed, if there is a vertex $v$ in $A_2\cap A_3$, then $u_1u_2vu_3u_1$ forms a cycle of length 4, contradicting the fact that $H$ has girth at least 5. 
    \begin{claim}
        For any vertex $v\in A_2$, there is a unique vertex $\phi (v)\in A_3$, such that $v$ is adjacent to $\phi (v)$. Moreover, for any two distinct vertices $v_1,v_2$ in $A_2$, we have $\phi (v_1)\ne \phi (v_2)$.
    \end{claim}
    \begin{proof}
        For any vertex $v\in A_2$, $A_2\cap A_3=\emptyset$ implies $u_3v\notin E(H)$.
        If $v$ is adjacent to $u_1$, then $u_1u_2v$ forms a triangle, which is a contradiction.
        So, $vu_1\notin E(H)$.
        By the assumption that $H$ is a strongly regular graph $H$ with parameters $(n,d,0,1)$, $u_3$ and $v$ have exactly one common neighbor, which must be in $A_3$.
        Denote it by $\phi(v)$.

        If there are two distinct vertices $v_1,v_2$ in $A_2$ such that $\phi (v_1)= \phi (v_2)$, then $u_2v_1\phi (v_1)v_2u_2$ forms a cycle of length 4, a contradiction.
        This completes the proof.
    \end{proof}
    For any $v\in A_2$, let $\psi_2(v)$ be the vertex in $G$ corresponding to the edge $u_2v$ in $H$. 
    For any $v\in A_3$, let $\psi_3(v)$ be the vertex in $G$ corresponding to the edge $u_3v$ in $H$.
    For any $p\in \left[ \frac{1}{2d-1},1 \right]$, let $\pi_p: V\times V\to [0,1]$ be the map defined as follows:
    \begin{center}
    $\pi_p(z,w)=\begin{cases}
    p-\frac{1-p}{2d-2}, &{\rm if}\ z=x, w=y;\\
    \frac{1-p}{2d-2}, &{\rm if}\ z=w,z\in \{x,y \};\\
    \frac{1-p}{2d-2}, &{\rm if}\ z=w,zx\in E, zy\in E;\\
    \frac{1-p}{2d-2}, &{\rm if}\ \exists \ v\in A_2, z=\psi_2(v), w=\psi_3(\phi(v));\\
    0, &{\rm otherwise}.
    \end{cases}$
    \end{center}
    By the Claim, it is straightforward to check that $\pi$ is a transport plan from $\mu_x^p$ to $\mu_y^p$.
    For any $v\in A_2$, since $v$ is adjacent to $\phi(v)$ in $H$, we have $d(\psi_2(v),\psi_3(\phi(v)))=2$.
    Thus, we calculate that
    \begin{align*}
        W_1(\mu_x^p,\mu_y^p)&=\inf_{\pi}\sum_{y\in V}\sum_{x\in V}d(x,y)\pi(x,y)\\
        &\le \sum_{y\in V}\sum_{x\in V}d(x,y)\pi_p(x,y)\\
        &= (p-\frac{1-p}{2d-2})+2|A_2|\times\frac{1-p}{2d-2}\\
        &= (p-\frac{1-p}{2d-2})+2(d-1)\times\frac{1-p}{2d-2}.
    \end{align*}
    It follows that 
    \begin{align*}
        \kappa(x,y)=\lim_{p\to 1}\frac{1-W_1(\mu_x^p,\mu_y^p)}{1-p}\ge \frac{1}{2d-2}.
    \end{align*}
    Recall that $\kappa(x,y) \le 1/(2d-2)$. This completes the proof.
\end{proof}

Now, we are prepared to prove Theorem \ref{ARTG}.
\begin{proof}[Proof of Theorem \ref{ARTG}]
    Let $G=(V,E)$ be an amply regular Terwilliger graph with parameters $(n, d, \alpha, \beta)$ such that $G$ has positive Lin--Lu--Yau curvature.
    By Lemma \ref{ARTG_up}, we have
    \begin{align*}
        0< \min_{xy\in E} \kappa(x,y)\le \frac{2\alpha +3-d}{d}.
    \end{align*}
    Therefore, we have $2\alpha +3-d >0$, which implies that
    \begin{align}\label{dalpha}
        d\le 2\alpha +2
    \end{align}

    The proof now splits into two cases, depending on whether $\beta > 1$.

    \noindent\textbf{Case 1:} $\beta = 1$.
    
    By Theorem \ref{lem:d-lowbdd} (i), we have $d\ge 2\alpha-\beta +3= 2\alpha +2$. Together with inequality \eqref{dalpha}, we find $d= 2\alpha +2$. Then, by Theorem \ref{lem:d-lowbdd} (ii), we know that $G$ is either the icosahedron or the line graph of a regular graph of girth at least $5$.
    Since the icosahedron has $\beta=2$, we conclude that $G$ is the line graph of a regular graph $H$ of girth at least $5$.
    By Lemma \ref{line} and the assumption that $G$ has positive Lin--Lu--Yau curvature, the diameter of $H$ is at most 2.
    Since $H$ has girth at least $5$, if $H$ has diameter 1, then $H$ is the complete graph $K_1$ or $K_2$, and hence $G$ is $K_0$ or $K_1$, which is a contradiction.
    Now, assume that the diameter of $H$ is 2.
    Since $H$ is a regular graph of girth at least $5$, we observe that $H$ is a strongly regular graph with parameters $(*,*,0,1)$.
    According to Lemma \ref{3.3}, $G$ has positive Lin--Lu--Yau curvature.
    Moreover, by Theorem \ref{thm:HS_Moore_srg}, $H$ must be isomorphic to one of the following graphs: the pentagon, the Petersen graph, the Hoffman–Singleton graph, or a strongly regular graph with parameters $(3250,57,0,1)$.

    \noindent\textbf{Case 2:} $\beta > 1$.

    By Theorem \ref{thm:1.16.3}, for any $\gamma \in G$, the reduced graph $\overline{G(\gamma)}$ of $G(\gamma)$ is a strongly regular Terwilliger graph with parameters
    \begin{align*}
    \bar{n} &= \frac{d}{s}, \quad \bar{d} = \frac{\alpha - s + 1}{s}, \quad \bar{\beta} = \frac{\beta - 1}{s}, \\
    \bar{\alpha} &= \frac{(\alpha - s + 1)(\alpha - 2s + 1) - (\beta - 1)(d - \alpha - 1)}{s(\alpha - s + 1)},
    \end{align*}
    where $s$ is the size of the equivalence classes of $G(\gamma)$.
    By inequality \eqref{dalpha}, we have $d\le 2\alpha +2$.
    Observe that
    \begin{align}\label{n1}
        \bar{n}= \frac{d}{s}\le  \frac{2\alpha +2}{s}= 2(\bar{d} +1).
    \end{align}
    Since $\overline{G(\gamma)}$ is strongly regular, Theorem \ref{lem:param_relation} yields
    \begin{align}\label{n2}
         \bar{n} - \bar{d} - 1 = \frac{\bar{d}(\bar{d} - \bar{\alpha} - 1)}{\bar{\beta}}.
    \end{align}
    Combining inequalities \eqref{n1} and \eqref{n2}, we deduce that
    \begin{align*}
        \frac{\bar{d}(\bar{d} - \bar{\alpha} - 1)}{\bar{\beta}}= \bar{n} - \bar{d} - 1\le \bar{d}+ 1.
    \end{align*}
    Thus, we have
    \begin{align*}
        \bar{d} - \bar{\alpha} - 1\le \frac{\bar{\beta}}{\bar{d}} (\bar{d}+ 1)= \bar{\beta}+ \frac{\bar{\beta}}{\bar{d}}\le \bar{\beta}+1,
    \end{align*}
    which implies
    \begin{align}\label{d1}
        \bar{d}\le \bar{\alpha}+ \bar{\beta}+2.
    \end{align}
    Recall from Theorem \ref{lem:d-lowbdd} (i) that
    \begin{align}\label{d2}
        \bar{d}\ge 2\bar{\alpha}- \bar{\beta}+3.
    \end{align}
    Combining inequalities \eqref{d1} and \eqref{d2} leads to
    \begin{align*}
        \bar{\alpha}\le 2\bar{\beta} -1.
    \end{align*}
    Substituting it into inequality \eqref{d1}, we obtain
    \begin{align*}
        \bar{d}\le 3\bar{\beta}+1.
    \end{align*}
    Substituting it into inequality \eqref{n1}, we derive
    \begin{align}
        \bar{n}\le 6\bar{\beta}+4< 50 \bar{\beta}.
    \end{align}
    According to Lemma \ref{cor:1.16.6}, $\overline{G(\gamma)}$ is the pentagon or the Petersen graph.
    If $\overline{G(\gamma)}$ is the Petersen graph, then we have $\bar{n}=10$ and $\bar{d}=3$, violating the inequality \eqref{n1}.
    Thus, $\overline{G(\gamma)}$ must be the pentagon.
    So, $\bar{\alpha}=0$. 
    According to Theorem \ref{thm:1.16.3}, we have 
    $s \leq \bar{\alpha} + 1 =1$.
    Therefore, $s=1$, which implies that $\overline{G(\gamma)}=G(\gamma)$.
    By the arbitrariness of the vertex $\gamma$, $G$ is locally pentagon.
    Note that the only graph that is locally a pentagon is the icosahedron (see, e.g., \cite[Proposition 1.1.4]{BCN89}).
    The proof is done.
\end{proof}

A direct corollary of Theorem \ref{ARTG} is the classification of all amply regular graphs with parameters $(n, d, \alpha, 1)$ and positive Lin--Lu--Yau curvature.

\begin{corollary}\label{mu1}
    Let $G$ be an amply regular graph with parameters $(n, d, \alpha, 1)$. If $G$ has positive Lin--Lu--Yau curvature, then $G$ is isomorphic to one of the following graphs: the pentagon, the line graph of the Petersen graph, the line graph of the Hoffman–Singleton graph, or the line graph of a strongly regular graph with parameters $(3250,57,0,1)$.
\end{corollary}

\begin{proof}
    The desired classification follows by Lemma \ref{beta=1} and Theorem \ref{ARTG}.
\end{proof}

\section{Classification}\label{section4}
In this section, we prove Theorem \ref{thm3}.
We first present a sharp lower bound for the normalized Laplacian eigenvalue $\lambda_1$ of a strongly regular graph. 

\begin{lemma}\label{thm:srg_eigenvalue}
    Let $G$ be a strongly regular graph with parameters $(n,d,\alpha,\beta)$ which is not the pentagon. Then we have
    $$\lambda_1\ge \frac{2+\alpha}{d},$$
    with equality if and only if $d=2\alpha -\beta +4$. 
\end{lemma}
\begin{proof}
We first show that if $\lambda_1\le (2+\alpha)/d$, then $d= 2\alpha -\beta +4$.
Assume that $\lambda_1\le (2+\alpha)/d$.
According to Theorem \ref{lam}, we obtain
\begin{align}\label{4.1}
  2d-3\alpha+\beta-4\le \sqrt{(\alpha-\beta)^2+4(d-\beta)}.
\end{align}
%By Lemma \ref{lem:d-lowbdd}(i), we have $d\ge 2\alpha -\beta +3$. 
Since $G$ has diameter $2$ and is not the pentagon, we conclude that $d\ge 2\alpha -\beta +4$ by Theorem \ref{lem:d-lowbdd}(ii).
It follows that
\begin{align}\label{positive}
    2d-3\alpha+\beta-4\ge (2\alpha -\beta +4)+ d-3\alpha+\beta-4 =d-\alpha \ge0.
\end{align}
Thus, we can square both sides of the inequality \eqref{4.1}. This yields
\begin{align}\label{3}
    (2d-3\alpha+\beta-4)^2\le (\alpha-\beta)^2+4(d-\beta).
\end{align}
The above inequality can be rearranged as
$$4(d-\alpha-1)(d-2\alpha+\beta-4)\le0.$$
Since $G$ is not complete, we observe that $d\ge\alpha+2$. Therefore, we derive
\begin{align*}
  d\le 2\alpha-\beta+4.
\end{align*}
Together with the lower bound $d\ge 2\alpha -\beta +4$, we find $d= 2\alpha -\beta +4$.

We next show that if $d= 2\alpha -\beta +4$, then $\lambda_1= (2+\alpha)/d$.
Assume that $d= 2\alpha -\beta +4$.
Then inequality \eqref{3} holds with equality, and hence inequality \eqref{4.1} also holds with equality (recall from inequality \eqref{positive} that $2d-3\alpha+\beta-4\ge 0$).
Therefore, we have $\lambda_1= (2+\alpha)/d$.
\end{proof}

The following basic lemma describes how the adjacency spectrum of a graph transforms under taking its line graph.

\begin{lemma}[{\cite[Lemma 8.2.5]{GR01}}]\label{spec}
Let $H$ be a connected $d$-regular graph with $n$ vertices and $m$ edges. Let
\[
d= \mu_1 > \mu_2 \ge \cdots \ge \mu_n
\]
be the adjacency eigenvalues of $H$. Then the adjacency eigenvalues of the line graph $G$ of $H$ are:
\begin{itemize}
    \item $\theta_i = \mu_i + d - 2$ for $i = 1, 2, \dots, n$, and
    \item $\theta = -2$ with multiplicity $m - n$.
\end{itemize}
\end{lemma}

\begin{theorem}\label{ex}
    An amply regular graph with parameters $(n, d, \alpha, 1)$ is not Lichnerowicz sharp.
\end{theorem}

\begin{proof}
    Assume, towards a contradiction, that $G=(V,E)$ is an amply regular graph with parameters $(n, d, \alpha, 1)$ which is Lichnerowicz sharp.
    Since $\lambda_1>0$, $G$ has positive Lin--Lu--Yau curvature.
    By Corollary \ref{mu1}, $G$ is the line graph of a strongly regular graph $H$ with parameters $(n',d',0,1)$. Then, we have $n= n'd'/2$, $d=2d'-2$, and $\alpha=d'-2$.
    By Theorem \ref{lam}, the second largest adjacency eigenvalue of $H$ is $(-1+\sqrt{4d'-3})/2$.
    Using Lemma \ref{spec}, we deduce that the second largest adjacency eigenvalue of $G$ is $d'-2+(-1+\sqrt{4d'-3})/2$.
    So, the smallest non-zero normalized Laplacian eigenvalue of $G$ is 
    \begin{align}\label{lamda1}
        \lambda_1=\frac{2d'+1-\sqrt{4d'-3}}{4d'-4}.
    \end{align}
    By Lemma \ref{3.3} and the assumption that $G$ is Lichnerowicz sharp, we have
    \begin{align}\label{lamda2}
        \lambda_1= \min_{xy\in E} \kappa(x,y)=\frac{1}{2d'-2}.
    \end{align}
    Combining equations \eqref{lamda1} and \eqref{lamda2}, we find $d'=1$, which is a contradiction.
\end{proof}

Next, we prove Theorem \ref{thm3}.
\begin{proof}[Proof of Theorem \ref{thm3}]
Let $G=(V,E)$ be a Lichnerowicz sharp distance-regular graph.
By Table \ref{table:Lichsharp_examples}, the complete graphs (i.e., $H(1,n)$, $n\geq 2$) are Lichnerowicz sharp.
Now, assume that $G$ is not complete.
According to Theorem \ref{ex}, we have $\beta >1$.
By Lemma \ref{up}, we deduce that
\begin{align}\label{1.4}
    \lambda_1=\min_{xy\in E}\kappa(x,y)\leq \max_{xy\in E}\kappa(x,y) \le \frac{2+\alpha}{d}.
\end{align}

\begin{table}[!ht]
\centering
\begin{tabular}{l|l|l|l|l|l}
$G$ & $|V|$ & Diameter & $\lambda_1$ & $\displaystyle{\inf_{xy\in E} \kappa(x,y)}$ & $\displaystyle{\frac{2+\alpha}{d}}$ \\[.1cm]
\hline &&&&& \\[-.2cm]

$H(d,n),d\ne 2,n\ge 2$ & $n^d$ & $d$ & $\frac{n}{d(n-1)}$ & $\frac{n}{d(n-1)}$ & $\frac{n}{d(n-1)}$ \\[.1cm]

${\rm Doob}^{n,m},n\ge 1,m\ge 1$ & $4^{n+2m}$ & $n+2m$ & $\frac{4}{3(n+2m)}$ & $\frac{2}{3(n+2m)}$ & $\frac{2}{3(n+2m)}$ \\[.1cm]

% $(7,2)$-Kneser & $21$ & $\min\{2,5\}=2$ & $\frac{7}{10}$ & $\frac{1}{2}$ & -- \\[.1cm]

Conway-Smith & $63$ & $3$ & $\frac{1}{2}$ & $-\frac{1}{10}$ & $\frac{1}{2}$ \\[.1cm]

Doro & $65$ & $3$ & $\frac{1}{2}$ & $-\frac{1}{10}$ & $\frac{1}{2}$ \\[.1cm]

$J(n,k),3\le k\le n-3$ & ${\binom{n}{k}}$ & $\min\{k,n-k\}$ & $\frac{n}{k(n-k)}$ & $\frac{n}{k(n-k)}$ & $\frac{n}{k(n-k)}$ \\[.1cm]

$Q^n_{(2)},n\ge 6$ & $2^{n-1}$ & $\left\lfloor \frac{n}{2}\right\rfloor$ & $\frac{4}{n}$ & $\frac{4}{n}$ & $\frac{4}{n}$ \\[.1cm]

Gosset & $56$ & $3$ & $\frac{2}{3}$ & $\frac{2}{3}$ & $\frac{2}{3}$\\[.1cm]

Icosahedron & $12$ & $3$ & $1 - \frac{\sqrt{5}}{5}$ & $\frac{2}{5}$ & $\frac{4}{5}$\\[.1cm]

\end{tabular}
\caption{Distance-regular graphs with second largest eigenvalue $\theta_1\geq b_1-1$, diameter not equal to 2, and $\beta \geq 2$}
\label{table:Lichsharp_examples}
\end{table}

We divide the proof into two cases according to whether the diameter of $G$ is at least three.

\noindent\textbf{Case 1:} $G$ has diameter 2 (that is, $G$ is a strongly regular graph).

Since $\beta >1$, $G$ is not the pentagon.
By Lemma \ref{thm:srg_eigenvalue} and inequality \eqref{1.4}, we have $\lambda_1=(2+\alpha)/d$. Applying Theorem \ref{lam}, we derive that the least adjacency eigenvalue of $G$ is $-d+2\alpha-\beta+2$. Inserting $d=2\alpha-\beta+4$ obtained from Lemma \ref{thm:srg_eigenvalue}, we conclude that the least adjacency eigenvalue of $G$ is $-2$.
Then, $G$ is one of the graphs listed in Theorem \ref{thm:1.1.1}.
By Table \ref{table1}, we check that $G$ is one of the following graphs: the cocktail party graphs $CP(n)$ for $n\ge 2$, the Hamming graphs $H(2,n)$ for $n\ge 2$, the triangular graphs $T(n)$ for $n\ge 4$, the Clebsch graph, and the Schl\"afli graph.

\begin{table}[!ht]
\centering
\begin{tabular}{l|l|l|l|l}
$G$ & $(n,d,\alpha,\beta)$ & $\lambda_1$ & $\displaystyle{\inf_{xy\in E} \kappa(x,y)}$ & $\displaystyle{\frac{2+\alpha}{d}}$ \\[.1cm] 
\hline &&&& \\[-.2cm]

$CP(n), n\geq 2$ & $(2n,2n-2,2n-4,2n-2)$ & $1$ & $1$ & $1$ \\[.1cm]

$H(2,n), n\geq 2$ & $(n^2,2(n-1),n-2,2)$ & $\frac{n}{2(n-1)}$ & $\frac{n}{2(n-1)}$ & $\frac{n}{2(n-1)}$ \\[.1cm]

$T(n), n\geq 4$ & $\left({\binom{n}{2}},2(n-2),n-2,4\right)$ & $\frac{n}{2(n-2)}$ & $\frac{n}{2(n-2)}$ & $\frac{n}{2(n-2)}$ \\[.1cm]

Shrikhande & $(16,6,2,2)$ & $\frac{2}{3}$ & $\frac{1}{3}$ & $\frac{2}{3}$ \\[.1cm]

Chang & $(28,12,6,4)$ & $\frac{2}{3}$ & $\frac{1}{3}$ & $\frac{2}{3}$ \\[.1cm]

Petersen & $(10,3,0,1)$ & $\frac{2}{3}$ & $0$ & $\frac{2}{3}$ \\[.1cm]

Clebsch ($Q^5_{(2)}$) & $(16,10,6,6)$ & $\frac{4}{5}$ & $\frac{4}{5}$ & $\frac{4}{5}$ \\[.1cm]

Schl\"afli & $(27,16,10,8)$ & $\frac{3}{4}$ & $\frac{3}{4}$ & $\frac{3}{4}$ \\[.1cm]
\end{tabular}
\caption{Strongly regular graphs with least adjacency eigenvalue $-2$}
\label{table1}
\end{table}

\noindent\textbf{Case 2:} $G$ has diameter at least 3.

By Table \ref{table:Lichsharp_examples}, the icosahedron is not Lichnerowicz sharp.
Therefore, $G$ is not the icosahedron.
According to Theorem \ref{thm:4.4.3}, we have $b^+:=\frac{b_1}{\theta_1+1}\geq 1$. Noticing that 
\[b^+=\frac{b_1}{\theta_1+1}=\frac{d-\alpha-1}{1+d-d\lambda_1},\]
we have 
$\lambda_1\ge (2+\alpha)/d$.
Together with inequality \eqref{1.4}, we find $\lambda_1=(2+\alpha)/d$, which is equivalent to $\theta_1=b_1-1$.
Then, $G$ is one of the graphs listed in Theorem \ref{thm:4.4.11} (note that the cases (i) and (ii) have been excluded by the assumption that $G$ has diameter at least 3 and $\beta >1$). By Table \ref{table:Lichsharp_examples}, we check that $G$ is one of the following graphs: the Hamming graphs $H(d,n)$ for $d\ge 3$ and $n\ge 2$, the Johnson graphs $J(n,k)$ for $3\le k\le n-3$, the demi-cubes $Q^n_{(2)}$ for $n\ge 6$, and the Gosset graph.
\end{proof}

\begin{remark}
    We notice that the triangular graphs $T(n)$ are special Johnson graphs $J(n,2)$. The Johnson graphs $J(n,1)=J(n,n-1)$ are the complete graphs $K_n$. The demi-cubes $Q^2_{(2)}$ and $Q^3_{(2)}$ are the complete graphs $K_2$ and $K_4$, respectively.
    The demi-cube $Q^4_{(2)}$ is the cocktail party graph $CP(4)$. The demi-cube $Q^5_{(2)}$ is the Clebsch graph.
\end{remark}

A direct corollary of Theorem \ref{thm3} is the classification of all Lichnerowicz sharp strongly regular graphs.

\begin{corollary}\label{srg}
   The Lichnerowicz sharp strongly regular graphs  are precisely the following ones: the cocktail party graphs $CP(n)$ for $n\ge 2$, the Hamming graphs $H(2,n)$ for $n\ge 2$, the triangular graphs $T(n)$ for $n\ge 4$, the Clebsch graph, and the Schläfli graph.
\end{corollary}

\iffalse
We recall the following classification result from \cite{CKKLMP20}.
      \begin{theorem} \cite[Theorem  6.4]{CKKLMP20} \label{LS-2}
    The Lichnerowicz sharp strongly regular graphs with $\lambda_1=\frac{2+\alpha}{d}$ are precisely the following ones: The
  cocktail party graphs $CP(n)$, $n \ge 2$, the lattice graphs
  $L_2(n) \cong K_n \times K_n$, $n \ge 3$, the triangular graphs
  $T(n) \cong J(n,2)$, $n \ge 5$, the demi-cube $Q^5_{(2)}$, and the
  Schl\"afli graph.
\end{theorem}
\fi

\section*{Acknowledgement}
We are grateful to the anonymous referees for their valuable suggestions.
We sincerely thank Florentin M\"{u}nch and one of the referees for pointing out the overlap between our results and M\"{u}nch's results on effective Bonnet--Myers sharp graphs.
We also thank Paul Terwilliger for pointing out the connection between Lichnerowicz sharp distance-regular graphs and the complete connected hypermetric spaces.
This work is supported by the National Key R \& D Program of China 2023YFA1010200 and the National Natural Science Foundation of China No. 12431004. K.C.'s research is supported by the National Natural Science Foundation of China No. 125B1009 and the New Lotus Scholars Program No. PB22000259. 

\section*{Appendix}
In the appendix, we provide references and explanations for the Lin--Lu--Yau curvature
values appearing in Tables \ref{table:Lichsharp_examples} and \ref{table1}.

For the cocktail party graphs \(CP(n)\), \(n\geq 2\), the Johnson graphs
\(J(n,k)\), \(1\leq k\leq n-1\), the demi-cubes \(Q^n_{(2)}\), \(n\geq 3\),
and the Gosset graph, the corresponding Lin--Lu--Yau curvature values were
computed in the extended arXiv version of \cite{CKKLMP20}; see Lemmas
4.1--4.4 therein. We also recall that the triangular graph \(T(n)\) is
isomorphic to the Johnson graph \(J(n,2)\).

The complete graph \(K_n\) on \(n\) vertices has constant Lin--Lu--Yau
curvature \(n/(n-1)\); see \cite[Example 1]{LLY11}, and also the more detailed
discussion in \cite[Section 3]{Bonini}. The Shrikhande graph \(Shk\) has
constant Lin--Lu--Yau curvature \(1/3\); see \cite[Comment after Theorem
1.6]{Bonini}.

The Hamming graph
\(
        H(d,n)=(K_n)^d, d\geq 1,\ n\geq 2,
\)
is the Cartesian product of \(d\) copies of \(K_n\). The Doob graph
\(
        \mathrm{Doob}^{n,m}=(K_4)^n \square (Shk)^m, n\geq 1,\ m\geq 1,
\)
is the Cartesian product of \(n\) copies of \(K_4\) and \(m\) copies of the
Shrikhande graph. Their Lin--Lu--Yau curvatures can be computed directly from
the following product formula of Lin, Lu and Yau \cite[Theorem 3.1]{LLY11}.

Recall that the Cartesian product \(G\square H\) of two graphs \(G\) and \(H\)
is the graph with vertex set \(V(G)\times V(H)\), where two vertices
\((u_1,v_1)\) and \((u_2,v_2)\) are adjacent if either
\(u_1=u_2\ \text{and}\ v_1v_2\in E(H),
\)
or
\(u_1u_2\in E(G) \ \text{and} \ v_1=v_2.
\) Suppose that \(G\) is \(d_G\)-regular and \(H\) is \(d_H\)-regular. Then the
Lin--Lu--Yau curvature of \(G\square H\) is given by
\begin{align}
\kappa^{G\square H}\bigl((u_1,v),(u_2,v)\bigr)
&=
\frac{d_G}{d_G+d_H}\,\kappa^G(u_1,u_2), \label{eq:cart-prod-1}\\
\kappa^{G\square H}\bigl((u,v_1),(u,v_2)\bigr)
&=
\frac{d_H}{d_G+d_H}\,\kappa^H(v_1,v_2), \label{eq:cart-prod-2}
\end{align}
where \(u\in V(G)\), \(v\in V(H)\), \(u_1u_2\in E(G)\), and
\(v_1v_2\in E(H)\).

Applying \eqref{eq:cart-prod-1} and \eqref{eq:cart-prod-2}, together with the
curvature values for \(K_n\) and \(Shk\), we obtain that the Hamming graph
\(H(d,n)\) has constant Lin--Lu--Yau curvature
\[
        \frac{n}{d(n-1)}.
\]
Similarly, the Doob graph \(\mathrm{Doob}^{n,m}\) has Lin--Lu--Yau curvature
values
\[
        \left\{
        \frac{4}{3(n+2m)},\,
        \frac{2}{3(n+2m)}
        \right\}.
\]

The Petersen graph is a strongly regular graph of girth \(5\), and has constant
Lin--Lu--Yau curvature \(0\); see, for example, \cite[Theorem 1.4]{Bonini}.

The Lin--Lu--Yau curvatures of the icosahedron, the Doro graph, the Chang
graphs, and the Schl\"afli graph can be computed directly from the formula of
M\"unch and Wojciechowski \cite[Theorem 2.6]{MW19}; see also
\cite[Theorem 1.1]{Hehl}.

For a given graph, the Lin--Lu--Yau curvature can also be computed using the
online Graph Curvature Calculator, available at
\url{https://www.mas.ncl.ac.uk/graph-curvature/}; see also
\url{https://skamtue.github.io/graph-curvature-webapp/}. By entering the
adjacency matrix of the graph, one obtains the Lin--Lu--Yau curvature of each
edge. Adjacency matrices of many distance-regular graphs are available from
DistanceRegular.org \cite{DistanceRegular}.

The online Graph Curvature Calculator is implemented in Python, and its design
is discussed in detail in \cite[Section 3]{CKLLS22}.

\end{document}